\documentclass[11pt]{amsart}
\usepackage[margin=2.5cm]{geometry}
\usepackage{epsfig}
\usepackage{graphicx}
\newtheorem{theorem}{Theorem}[section]
\newtheorem{proposition}{Proposition}[section]
\newtheorem{lemma}{Lemma}[section]
\newtheorem{remark}{Remark}[section]
\newcommand{\NX}{N_{\mathcal{X}}}
\newcommand{\brref}[1]{(\ref{#1})}

\newcommand{\gmdp}{
\author[G. Besana]{ Gian Mario Besana}
\address {Gian Mario Besana \\ College of Computing and Digital Media\\ DePaul University\\1 E. Jackson blvd\\Chicago IL 60604 USA}
 \email{gbesana@depaul.edu}}

\newcommand{\Pin}[1]{{\mathbb P}^{#1}}

\newcommand{\lra}{\longrightarrow}

\newcommand{\rk}[1]{{\rm rk}\,(#1)}

\newcommand{\bxi}{\{\mathbf{X}_i\}}
\newcommand{\bxj}{\{\mathbf{X}_j\}}
\newcommand{\nbxi}{\mathbf{X_i}}
\newcommand{\nbXi}{\mathbf{X}_i}

\newcommand{\codim}{codim}
\title[Matrices dropping rank in codim $1$]{Matrices dropping rank in codimension one and critical loci in computer vision}
\author[M. Bertolini]{Marina Bertolini}
\address{Marina Bertolini\\ Dip. di Matematica F. Enriques\\ Universit\'a degli studi di Milano\\ Milano, Italy.}
\email{marina.bertolini@unimi.it}
\gmdp
\author[R. Notari]{Roberto Notari}
\address{Roberto Notari\\Dip. di Matematica F. Brioschi, Politecnico di Milano, Milano, Italy.}
\email{roberto.notari@polimi.it}
\author[C. Turrini]{Cristina Turrini}
\address{Cristina Turrini\\ Dip. di Matematica F. Enriques\\ Universit\'a degli studi di Milano\\ Milano, Italy.}
\email{cristina.turrini@unimi.it}
\begin{document}
\begin{abstract}
  Critical loci for projective reconstruction from three views in four dimensional projective space are defined by an ideal generated by maximal minors of suitable $4 \times 3$ matrices, $N,$ of linear forms. Such loci are classified in this paper, in the case in which $N$ drops rank in codimension one, giving rise to reducible varieties. This leads to a complete classification of matrices of size $(n+1) \times n$ for $n \le 3,$ which drop rank in codimension one. Instability of reconstruction near non-linear components of critical loci is explored experimentally.
\end{abstract}
\maketitle


\section{Introduction}

Linear projections from $\Pin{k}$ to $\Pin{2},$ and even
from $\Pin{k}$ to $\Pin{m}, m\ge 3,$ are of interest to the
computer vision community as simple models of the process of taking pictures of particular three-dimensional scenes,
see for example \cite{Shas-Wo}, \cite{Hart-Schaf},
\cite{hua-fos-ma}, \cite{vid2}, \cite{vid4}, \cite{vid3},
\cite{vid1}.
 In this framework, the authors have investigated in several works, \cite{be-tur1}, \cite{tubbCHAPTER}, \cite{tubbICCV07},
 \cite{tubbISVC08}, \cite{tubbMagic}, \cite{LAIA}, \cite{bntMEGA}, the algebro-geometric properties of varieties that arise as {\it critical loci} for projective reconstruction from multiple views, i.e. from multiple projections $\Pin{k} \dasharrow \Pin{2}$ (see Sections \ref{prelimCV} and \ref{critlocsetup} for an introduction to the reconstruction problem and the notion of critical locus). In particular, \cite{LAIA} presents a comprehensive treatment of the  families of varieties involved, if the number of views is the minimum necessary for reconstruction. In \cite{LAIA} it is shown that, under suitable genericity assumptions, critical loci are either hypersurfaces if the ambient space is odd dimensional, or special determinantal varieties of codimension two if the ambient space is even dimensional. The latter codimension-two varieties are studied in great detail in \cite{bntMEGA}, in the case of $3$ projections (views) from $\Pin{4}$ as ambient space, where they are shown to fill the irreducible component of the Hilbert scheme of $\Pin{4}$ whose general element is a classical Bordiga surface.

 When genericity assumptions are dropped, one is led to consider a number of degenerate configurations of centers of
 projection and corresponding degenerate critical loci. In the classical case of projections from $\Pin{3}$ to $\Pin{2},$
 Hartley and Kahl described degenerate cases in \cite{Hart-Ka}. This work revisits the case of $3$ projections from $\Pin{4}$
 to $\Pin{2},$ while dropping the genericity assumptions, in order to conduct a detailed analysis of possible degenerate cases of critical loci, in a proper algebro-geometric setting.

 In \cite{bntMEGA} it is shown that the minimal generators of the ideal of the critical locus for 3 projections
 from $\Pin{4}$ to $\Pin{2}$ are cubic polynomials, arising as maximal minors of a suitable $4 \times 3$ matrix $N$
 of linear forms. Genericity assumptions are reflected in the fact that such minors do not share any common factors.
 Given the focus of this work, one is naturally led to consider $(n+1) \times n$ matrices of linear forms, whose minors
 have common factors. Matrices of type $ (n+1) \times n $ that drop rank in codimension
$2$ have been intensively studied from a geometrical point of view in the framework of liaison
theory (e.g. see Gaeta's Theorem), while, in the framework of commutative algebra, many
researchers have contributed to generalize Hilbert--Burch Theorem
by deeply studying homological properties of rings and modules.


On the contrary, matrices of type $(n+1) \times n$ of linear forms
that drop rank in codimension $1$ do not seem to have been studied
with any systematic approach. In order to address the main goal of
this work, a classification of canonical forms of such matrices,
over any field, for $ n\leq 3$ is conducted. The authors believe
that this contribution may have relevance in other contexts and be
of more general interest. Theorem \ref{4x3-deg1} and Theorem
\ref{4x3-deg2} contain such classification, respectively for
minors with a linear or quadratic common factor. Degeneration loci
of such matrices are then investigated in Section \ref{DegLoc}, in the case of interest for this work, when the ambient
space is four-dimensional.
Leveraging this analysis, a complete geometrical classification of degenerate
critical loci for three projections from $\Pin{4}$ to $\Pin{2}$ is presented in Theorem \ref{crtiloctheo}. For degenerate critical loci in
situations in which the reconstruction of the trifocal tensor
still makes sense, see Section \ref{rec_tens_deg}, and the tensor can be uniquely reconstructed,
we conduct instability experiments, following \cite{tubbCHAPTER}.

The paper is organized as follows: Section \ref{34mat} contains
the classification of $4 \times 3$ matrices whose maximal minors
share a common factor, see Theorem \ref{4x3-deg1} and Theorem
\ref{4x3-deg2}. In Section \ref{DegLoc} the degeneracy loci of
such matrices are studied. Section \ref{prelimCV} introduces
multiview geometry and basic facts on the reconstruction problem
in computer vision. Section \ref{critlocsetup} presents the notion
of critical locus for three projections from four-dimensional
projective space, and, relating its study to matrices investigated
in Section \ref{34mat}, presents a classification of such loci,
see Theorem \ref{crtiloctheo}. In Section \ref{rec_tens_deg} we
address the actual possibility of reconstructing the trifocal
tensor in each situation appearing in Theorem \ref{crtiloctheo}.
Finally, in Section \ref{inst_exp}, we conduct experiments to
investigate the instability of the reconstruction algorithm in a
neighborhood of each non linear component of the critical loci
obtained in Theorem \ref{crtiloctheo}.

\section{Classification of $ 4 \times 3 $ matrices that drop rank in codimension $ 1 $}
\label{34mat}

In this section we compute the canonical forms of $(n+1) \times n$
matrices, with $n \leq 3$, of linear forms whose maximal minors
have a non trivial common factor, up to elementary operations on
rows and columns. As the maximal minors have degree at most three,
the degree of the common factor is either $1$ or $2$.

As standard notations, $R$ is the polynomial ring in $x_0, \dots,
x_r$ over any field $K$ with $char(K) \neq 2.$ Let $A$ and $B$ be matrices with entries in $R.$ $A^T$ denotes the transpose of $A.$ $A_{\widehat{i_1\dots i_r}}$ denotes the matrix obtained from $A$ by deleting the $i_1, \dots, i_r$ rows. Assuming that $A$ and $B$ have compatible sizes, $AB$ denotes the product of the two matrices, while $(A|B)$ denotes matrix concatenation. Let $f_1, \dots f_r$ be polynomials in $R;$ then $\langle f_1, \dots, f_r\rangle$ denotes the ideal generated by the ${f_i}$s.

In the rest of this work we will often deal with matrices whose entries are linear forms. Unless otherwise explicitly stated, we will always assume that linear forms appearing in columns of matrices involved in our arguments, are as linearly independent as possible, given explicit assumptions made in each instance.

\subsection{Matrices with maximal minors having a common factor of degree $1$}
\label{deg1}

In this subsection, we classify the $ 4 \times 3 $ matrices whose
maximal minors have a greatest common divisor of degree $ 1,$ see Theorem \ref{4x3-deg1}. To
achieve our goal, we analyze first matrices of type $ 2 \times 1 $
and $ 3 \times 2,$ see paragraph below and Proposition \ref{case3x2}. Along the way, we prove a general technical result, see Lemma \ref{case-(n+2)x(n)}.

Any $ 2 \times 1 $ matrix of linear forms whose maximal minors have a greatest common divisor of degree $ 1$ has entries that differ by a
multiplicative constant. Hence it can be reduced, via elementary row operations, to the form $$ \left( \begin{array}{c} 0 \\ \ell \end{array} \right). $$

\begin{lemma} \label{case-(n+2)x(n)}
Let $ M $ be a $ (n+2) \times n $ matrix of linear forms whose $
(n+1) \times n $ submatrices do not drop rank in codimension $ 1$.
Let $ P $ be a $ 1 \times (n+2) $ matrix of forms of degree $ n $
such that $ P M = 0$. Then, the forms in $ P $ are linearly
dependent, and a free resolution of $ P $ is $$ 0 \to
\begin{array}{c} R(-n) \\ \oplus \\ R^n(-n-1) \end{array}
\stackrel{\widetilde{M}}{\longrightarrow} R^{n+2}(-n)
\stackrel{P}{\longrightarrow} R \to 0,$$ where $\widetilde{M}$ is
the concatenation of $M$ and a suitable column of constants.
\end{lemma}

\begin{proof} Let $ D_{ij} = (-1)^{i+j} \det(M_{\widehat{ij}}) $ if $ i < j$, $ D_{ij} = -D_{ji} $ if $ i > j$ and
$ D_{ii} = 0.$
Let $ D $ be the order $ n+2 $ square matrix whose elements are $ D_{ij}$. By construction, $ D $ is skew--symmetric and
homogeneous of degree $ n$. Moreover, $ D M = 0.$  To see this, note that each row of $ D$
consists of maximal minors of $ M_{\widehat{i}},$ with the appropriate sign. As such a matrix does not drop rank in codimension $ 1$,
we can apply Hilbert--Burch Theorem to get the claim. $ D $ drops rank in codimension $ 3 $ and so
$$ 0 \to R^n(-n-2) \stackrel{M}{\longrightarrow} R^{n+2}(-n-1) \stackrel{D}{\longrightarrow} R^{n+2}(-1) \stackrel{M^T}
{\longrightarrow} R^n \to 0 $$ is a minimal free resolution of $ M^T.$

Let $ P $ be a $ 1 \times (n+2) $ matrix of forms of degree $ n $ such that $ P M = 0$. Then, there exists a
$ 1 \times (n+2) $
 constant matrix $ X $ such that $ P = X D$. As $ D $ is skew--symmetric, $ X D X^T = -X D X^T $ and so $ P X^T = 0$,
 that is to say, the elements of $ P $ are linearly dependent. Furthermore, $ P M = 0$, and so, letting
 $ \widetilde{M}= (X^T | M) $ , it is $ P \widetilde{M} = 0.$ Moreover, it is easy to check that the maximal minors of
 $ \widetilde{M} $ are equal to the elements of $ P$, up to sign. Hence, Hilbert--Burch Theorem gives the  free resolution of
 $ P$, as in the statement.
\end{proof}

The following additional lemma deals with $ 2 \times 2 $ matrices of linear forms, whose determinant is the product of linear forms.
\begin{lemma} \label{case-2x2}
Let $$ A = \left( \begin{array}{cc} a_{11} & a_{12} \\ a_{21} & a_{22} \end{array} \right) $$ be a square matrix of linear
forms such that $ \det(A) = uv $ for suitabe linear forms $ u,v$. Then, up to elementary operations on the columns of $ A$,
the elements of a column are linearly dependent.
\end{lemma}

\begin{proof} We can assume that $ a_{11}, a_{21} $ are linearly independent. The assumption $ \det(A)= uv $ can be rewritten
as $ -a_{11} a_{22} + a_{21} a_{12} + uv =0.$ First assume that  $ a_{11}, a_{21}, u $ are linearly independent.

Then $$ \left( \begin{array}{c} a_{22} \\ a_{12} \\ v \end{array} \right) = \left( \begin{array}{ccc} 0 & u & a_{21} \\ -u & 0 & a_{11} \\ a_{21} & a_{11} & 0 \end{array} \right) \left( \begin{array}{c} x_1 \\ x_2 \\ x_3 \end{array} \right) $$ for suitable constants $ x_1, x_2, x_3$, because, in this case,
the first syzygy module of $ -a_{11}, a_{21}, u $ is generated by the Koszul relations. Hence, $ a_{22} = x_2 u + x_3 a_{21} $ and $ a_{12} = -x_1 u + x_3 a_{11}$, that is to say, $$ A = \left( \begin{array}{cc} a_{11} & -x_1 u \\ a_{21} & x_2 u \end{array} \right) $$ after performing the elementary operation $ C_2 - x_3 C_1 $ on the second column of $ A$.

Now assume instead that $ u = \alpha a_{11} + \beta a_{21}, $ for suitable constants $ \alpha, \beta$.
The assumption $ \det(A)= uv $ can be rewritten as $ (-a_{22} + \alpha v) a_{11} + (a_{12} + \beta v) a_{21} = 0$.
Hence, there exists a constant $ \gamma $ such that $ a_{22} = \alpha v + \gamma a_{21}, a_{12} = -\beta v + \gamma a_{11} $
because the syzygies of $ a_{11}, a_{21} $ are again generated exclusively by the Koszul relations. Hence, after performing
the elementary operation $ C_2 - \gamma C_1 $ on the second column of $ A $, the elements (linear forms) of the second column of $ A $
are linearly dependent, as in the previous case.
\end{proof}

\begin{proposition} \label{case3x2}
Let $N$ be a $3 \times 2$ matrix of linear forms, whose maximal
minors have a greatest common divisor of degree $1$. Then, up to
elementary operations on rows and columns, it is:
\begin{equation} \label{3-2-form} N = \left( \begin{array}{cc} 0 & n_{12} \\
0 & n_{22} \\ n_{31} & n_{32} \end{array} \right) \qquad \mbox{or}
\qquad N = \left(
\begin{array}{cc} 0 & n_{12} \\ n_{12} & 0 \\ n_{31} & n_{32}
\end{array} \right) \end{equation}.
\end{proposition}

\begin{proof}

As the maximal minors of $N$ have a greatest common divisor of
degree $1$, there exists a $ 1 \times 3 $ matrix $ P $ of linear
forms such that $ P N = 0$. Of course, both columns of $ N $ are
syzygies of $ P$. Let $C$ be the first column of $ N.$

First assume that the linear forms in $C$ are linearly
independent. Then Lemma \ref{case-(n+2)x(n)} applied  to $ C $
implies that there exists $ X $ such that $ P = X D,$ and the syzygy matrix of $P$ is $(X^T | C).$
 The second column of $ N $ is then $ \alpha X^T +
\beta C $ where $ \alpha $ is now a linear form and $ \beta $ is a
constant. Therefore, up to elementary operations on rows and columns of
the matrix, $N$ is as in (\ref{3-2-form}), case on the left.

One can now assume that the linear forms in the first column $C$ span a
subspace of dimension $ 2.$
Then one can reduce $ N $ to the form
$$ N = \left( \begin{array}{cc} 0 & n_{12} \\ n_{21} & n_{22} \\ n_{31} & n_{32} \end{array} \right),$$
with $ n_{21}$ and $ n_{31} $ linearly independent, and $ n_{12} \not=
0$, to avoid trivial cases.

Note that we can also assume that the
linear forms in the second
column of $N$ span a space of dimension $2,$ otherwise $N$ would again be as in (\ref{3-2-form}), case on the left. Because, moreover, $n_{12} \not= 0$, we have that $ n_{22}, n_{32} $ are linearly
dependent modulo $ n_{12}$, and so we can further reduce $ N $ by
assuming that $ n_{22}=0.$ The maximal minors of $ N $ are $
n_{21} n_{32}, n_{12} n_{31}, -n_{21} n_{12}.$ Then, $ n_{12} $
divides $ n_{21} n_{32}$. We know that $ n_{12} $ and $ n_{32} $
are linearly independent, and so, $ n_{12} $ divides $ n_{21}$.
Hence $N$ is as in (\ref{3-2-form}), case on the right.
\end{proof}
\bigskip

\begin{theorem} \label{4x3-deg1}
Let $ N $ be a $ 4 \times 3 $ matrix of linear forms whose maximal
minors have a greatest common divisor of degree $1$. Then, up to
elementary operations on its rows and columns, $ N $ has one
of the following forms, where $\alpha$ and $\beta$ are suitable constants, $(\alpha, \beta) \neq (0,0)$:

\begin{equation} \label{4x3-1} N_A =
\left(
\begin{array}{ccc} n_{11} & n_{12} & 0 \\ n_{21} & n_{22} & 0 \\
n_{31} & n_{32} & 0
\\ n_{41} & n_{42} & n_{43} \end{array} \right),
\qquad
N_B = \left( \begin{array}{ccc} 0 & 0 & n_{13} \\ 0 & n_{13} & 0 \\
n_{31} & n_{32} & n_{33} \\ n_{41} & n_{42} & n_{43}
\end{array} \right), \end{equation}

\begin{equation} \label{4x3-2} N_C = \left( \begin{array}{ccc} 0 & 0 & \alpha n_{31} + \beta n_{22} \\
0 & n_{22} & \alpha n_{41} \\ n_{31} & 0 & \beta n_{42}
\\ n_{41} & n_{42} & 0 \end{array} \right), \qquad  N_D = \left(
\begin{array}{ccc} 0 & 0 & n_{13} \\ 0 & n_{31} & n_{23} \\ n_{31}
& 0 & n_{33} \\ n_{41} & n_{42} & n_{43}
\end{array} \right). \end{equation}
\end{theorem}

\begin{proof} Let $ N $ be a $ 4 \times 3 $ matrix of linear forms, whose
maximal minors have a greatest common divisor of degree $ 1$.
Then, up to the common linear factor, the maximal minors of $ N $ are $ 4
$ quadratic forms. Let $ P $ be the row vector whose entries are these quadratic forms, in the right order. Then $ P N = 0$, which shows that
$ N $ is a submodule of the syzygies of $P$.

Let us first assume that there exist two columns of $ N, $ say the first two $C_1$ and $C_2$,
with the property that every $ 3 \times 2 $ submatrix of $N(1,2) = (C_1|C_2)$ has maximal
minors with trivial greatest common divisor.
Then, as in the proof of  Lemma \ref{case-(n+2)x(n)},
it is $ P = X D $ for a suitable $ X $ of type $ 1 \times 4 ,$ and
the syzygy matrix of $ P $ is equal to $
(X^T | C_1 | C_2)$. Hence, the third column of $
N $ is $ C_3 = \delta X^T + \beta_1 C_1 + \beta_2 C_2 $ where $
\delta $ is a suitable linear form and $ \beta_1, \beta_2 $ are
constants. Thus, up to elementary operations on rows and columns, it is $N=N_A$ in (\ref{4x3-1}).
 Notice that if the linear
forms of any column of $ N $ span a $1$ dimensional subspace, then
the maximal minors of $N$ have a common factor of degree $1$ and
$N$ can be reduced as above. Therefore from now on we will assume
that the linear spaces spanned by the linear forms of each column
have dimension at least two.

From the above discussion, we can now assume that every pair of columns of $N$
has a $ 3 \times 2 $ submatrix whose maximal minors have
a greatest common divisor of degree $1.$ Proposition \ref{case3x2}, up to
elementary operations on rows and columns, gives the following possibilities for $N:$
$$ N_1 = \left( \begin{array}{ccc} 0 & 0 & n_{13} \\ 0 & 0 & n_{23} \\ n_{31} & n_{32} & n_{33} \\ n_{41} & n_{42} & n_{43} \end{array} \right),
N_2 = \left( \begin{array}{ccc} 0 & 0 & n_{13} \\ 0 & n_{13} & 0 \\ n_{31} & n_{32} & n_{33} \\ n_{41} & n_{42} & n_{43} \end{array} \right),
N_3 = \left( \begin{array}{ccc} 0 & 0 & n_{13} \\ 0 & n_{22} & n_{23} \\ n_{31} & n_{13} & 0 \\ n_{41} & n_{42} & n_{43} \end{array} \right),$$
$$ N_4 = \left( \begin{array}{ccc} 0 & n_{12} & n_{13} \\ 0 & n_{22} & n_{23} \\ n_{31} & n_{32} & 0 \\ n_{41} & 0 & n_{32} \end{array} \right),
N_5 = \left( \begin{array}{ccc} 0 & n_{12} & n_{31} \\ n_{12} & 0 & n_{32} \\ n_{31} & n_{32} & 0 \\ n_{41} & n_{42} & n_{43} \end{array} \right), N_6 = \left( \begin{array}{ccc} 0 & 0 & n_{13} \\ 0 & n_{22} & n_{23} \\ n_{31} & 0 & n_{33} \\ n_{41} & n_{42} & n_{43} \end{array} \right).$$

Let us consider $ N_1$. If $ n_{13} $ and $ n_{23} $ are linearly
dependent, then $ N_1 ,$ up to elementary
operations, has a row of zeros. Therefore we can assume that $ n_{13}$ and $n_{23}$ are
linearly independent, and so $ n_{31} n_{42} - n_{32} n_{41} $ is
reducible. Then, Lemma \ref{case-2x2}, up to elementary
operations, implies that the forms of the first or of the second
column of $ N_1 $ span a space of dimension $ 1$, and hence $ N_1
$ is a specialization of $ N_A$.

It is immediate to see that $ N_2 = N_B,$ and that its maximal minors have $n_{13}$ as a common factor.

Now we consider $N_3$. Its maximal minors have $n_{13}$ as
common factor if $n_{13}$ divides either $n_{31}$ or
$n_{22}n_{43}-n_{23}n_{42}.$ In the first case $N_3$ can be
reduced to $N_B$ with $n_{32}=0.$ In the second case, by
Lemma \ref{case-2x2}, $N_3$ can be reduced to the form
$$ \left( \begin{array}{ccc} 0 & 0 & n_{13} \\ 0 & n_{22} &
n_{23} \\ n_{31} & n_{13} & 0 \\ n_{41} &0 & 0
\end{array} \right)$$ that is a specialization of $ N_D$.

The maximal minors of $ N_4$, up to a sign, are $$ n_{32} \left(n_{41}n_{23} + n_{22}n_{31}\right), n_{32}\left(n_{41}n_{13}+n_{12}n_{31}\right), n_{41}\left(n_{12}n_{23}-n_{13}n_{22}\right), n_{31}\left(n_{12}n_{23}-n_{13}n_{22}\right).$$ As $ n_{31}$ and $ n_{41} $ are linearly independent, then either $ n_{32} $ divides $ n_{12}n_{23}-n_{13}n_{22}$, or $ n_{31} $ divides both $ n_{13} $ and $ n_{23}$, or $ n_{41} $ divides both $ n_{12} $ and $ n_{22}$, or, finally, $ n_{41}n_{23}+n_{22}n_{31}, n_{41}n_{13}+n_{12}n_{32},$ and $ n_{12}n_{23}-n_{13}n_{22} $ have a common linear factor. In the first case, a column of the submatrix formed by the first two rows and columns of $ N_4 $ contains forms that are multiple of $ n_{32}$, and so a column of $ N_4 $ spans a linear space of dimension $ 1$. Hence, we get a special case of $ N_A$. In the second case it is easy to check that we get a specialization of $ N_D$. The third case is identical to the previous one, swapping the role of $n_{31}$ and $n_{41}.$ In the last case, the maximal minors of the matrix $$ \left( \begin{array}{cc} n_{12} & n_{13} \\ n_{22} & n_{23} \\ -n_{41} & n_{31} \end{array} \right) $$ have a common linear form. By Proposition \ref{case3x2}, we get a specialization of either $ N_A $ or $ N_B$.

The maximal minors of $ N_5$, up to the sign, are $$ 2n_{12}n_{31}n_{32}, n_{12}\left(n_{32}n_{41}+n_{31}n_{42}-n_{12}n_{43}\right), n_{31}\left(-n_{32}n_{41}+n_{31}n_{42}-n_{12}n_{43}\right),$$ and $$ n_{32}\left(-n_{32}n_{41}+n_{31}n_{42}+n_{12}n_{43}\right).$$  Assume first that $ n_{12} $ is the common factor among them. Then, $ n_{12} $ divides $ -n_{32}n_{41}+n_{31}n_{42}$, that forces $ N_5 $ to be a specialization of $ N_A$. The cases when $n_{31}$ or $n_{32}$ is the common factor are dealt with similarly.

The last remaining case, $ N_6,$ requires a slightly more elaborate analysis.
The
maximal minors of $ N_6$, up to sign, are \begin{equation}
\label{lista}
n_{13}n_{31}n_{42}, n_{13}n_{22}n_{41}, n_{13}n_{22}n_{31},
n_{23}n_{31}n_{42}+n_{33}n_{22}n_{41}-n_{43}n_{22}n_{31}.
\end{equation} We
remark that $ n_{31}, n_{41}$ and $ n_{22}, n_{42} $ are pairwise
linearly independent. The maximal minors of $ N_6 $ have a common
linear form if one of the following condition holds:
\begin{itemize}
\item[(1)] $ n_{31} $ divides both $ n_{13} $ and $ n_{33};$
\item[(2)] $n_{31} $ divides $ n_{22};$
\item[(3)] $ n_{42} $ divides both $ n_{13} $ and $n_{33}n_{41}-n_{31}n_{43};$
\item[(4)] $ n_{13} $ divides $n_{23}n_{31}n_{42}+n_{33}n_{22}n_{41}-n_{43}n_{22}n_{31}.$
\end{itemize}
The cases above follow from the analysis of the maximal minor $n_{13}n_{31}n_{42}.$ An analogous list could be obtained starting from the second or third maximal minor in \brref{lista}.

In case $(1)$, $ N_6 $ is a specialization of $ N_B$.
In cases $(2)$ and $(3)$, $N$ is either $N_D$ or one of its specializations.
The rest of this proof is devoted to case $(4).$
First notice that we can assume that the submatrix obtained from the last $ 3 $ rows and first two
columns of $ N_6 $ satisfies the hypothesis of the Hilbert--Burch
Theorem. Otherwise we would fall back on case $(2).$
Case (4) can be rewritten as
\begin{equation}\label{eqq}
n_{23}\left(n_{31}n_{42}\right)+n_{33}\left(n_{22}n_{41}\right)-n_{43}\left(n_{22}n_{31}\right)
+ n_{13} q = 0
\end{equation} for a suitable quadratic form $ q$. Let $ I = \langle n_{31}n_{42}, n_{22}n_{41},
n_{22}n_{31}\rangle$, and $ J =\langle n_{13}\rangle$. Then $ I + J $ is generated by all the previous generators, and
\brref{eqq} gives that the transpose of $ (n_{23},
n_{33},-n_{43}, q) $ is a syzygy of the generators of $ I+J$.

\noindent Case $(4.1)$: $ n_{13} $ is regular for $ I$.

Then
 the syzygy matrix of $ I+J $ can be
computed by mapping cone procedure, and it is $$ M_1 = \left(
\begin{array}{ccccc} -n_{13} & 0 & 0 & 0 & n_{22} \\ 0 & -n_{13} &
0 & n_{31} & 0 \\ 0 & 0 & -n_{13} & -n_{41} & -n_{42} \\
n_{31}n_{42} & n_{22}n_{41} & n_{22}n_{31} & 0 & 0 \end{array}
\right).$$ Hence, there exist constants $ x_1, \dots, x_5 $ such
that $$ \left( \begin{array}{c} n_{23} \\ n_{33} \\ -n_{43} \\ q
\end{array} \right) = M_1 \left( \begin{array}{c} x_1 \\ \vdots \\
x_5 \end{array} \right)$$ that is to say $ n_{23} =
-x_1n_{13}+x_5n_{22}, n_{33} = -x_2n_{13}+x_4n_{31}, n_{43} =
x_3n_{13}+x_4n_{41}+x_5n_{42}$. By substituting in $ N_6 $ and
performing suitable elementary operations, we get a matrix that is
a specialization of $ N_A$.

We remark that $ I = \langle n_{31}, n_{41} \rangle \cap \langle
n_{22}, n_{42} \rangle \cap \langle n_{22}, n_{31} \rangle.$ So,
if $ n_{13} $ is not regular for $ I$, then it has to belong to
one of the three ideals whose intersection is $ I$.

\noindent Case $(4.2)$: $ n_{13} = \alpha n_{31} + \beta n_{41} $
with $ (\alpha, \beta) \not= (0,0)$.

The syzygy matrix of $ I+J$, in this case, is $$ M_2 = \left(
\begin{array}{cccc} 0 & 0 & n_{22} & -n_{13} \\ -\beta & n_{31} &
0 & 0 \\ -\alpha & -n_{41} & -n_{42} & 0 \\ n_{22} & 0 & 0 &
n_{31}n_{42} \end{array} \right).$$ Hence, there exists a linear
form $ v $ and constants $ x_2, x_3, x_4 $ such that $$ \left(
\begin{array}{c} n_{23} \\ n_{33} \\ -n_{43} \\ q
\end{array} \right) = M_2 \left( \begin{array}{c} v \\ x_2 \\
x_3 \\ x_4 \end{array} \right)$$ and so $ n_{23} = x_3n_{22} -
x_4n_{13}, n_{33} = -\beta v + x_2 n_{31}, n_{43} = \alpha v + x_2
n_{41} + x_3 n_{42}$. By substituting in $ N_6 $ and performing
suitable elementary operations, we get a matrix that is a
specialization of $ N_D$.

\noindent Case $(4.3)$: $ n_{13} = \alpha n_{22} + \beta n_{42} $
with $ (\alpha, \beta) \not= (0,0)$.

This case is completely analogous to $ (4.2)$ and leads to a specialization of $N_D.$

\noindent Case $(4.4)$: $ n_{13} = \alpha n_{31} + \beta n_{22} $
with $ (\alpha, \beta) \not= (0,0)$.

The syzygy matrix of $ I+J$, in this case, is $$ M_3 = \left(
\begin{array}{cccccc} 0 & n_{22} & n_{13} & 0 & 0 & \alpha n_{41}
\\ n_{31} & 0 & 0 & n_{13} & 0 & \beta n_{42} \\ -n_{41} & -n_{42}
& 0 & 0 & n_{13} & 0 \\ 0 & 0 & -n_{31}n_{42} & -n_{22}n_{41} &
-n_{22}n_{31} & -n_{41}n_{42} \end{array} \right).$$ There exist
constants $ x_1, \dots, x_6 $ such that, as before, $ n_{23} =
x_2n_{22}+x_3n_{13}+\alpha x_6 n_{41}, n_{33} =
x_1n_{31}+x_4n_{13}+\beta x_6 n_{42}, n_{43} = x_1 n_{41} + x_2
n_{42} - x_5 n_{13}$. By substituting in $ N_6 $ and performing
suitable elementary operations, we get the matrix $ N_C.$
\end{proof}

\begin{remark} Each of the matrices $ N_A, N_B, N_C, N_D $ is not a specialization of any of the others. In fact, in general cases,the  columns of each matrix span a space of maximal dimension, and such dimensions are uniquely associated to each matrix.
\end{remark}

\begin{remark} Note that each general matrix $ N_A, N_B, N_C, N_D $ can be specialized to a specializations of the remaining three. This statement can be interpreted from a different point of view. The set of $ 4 \times 3 $ matrices of linear forms over a given polynomial ring is an affine space, or a projective space, if we exclude the null matrix, and so it is irreducible. The general element of such a space corresponds to a matrix whose maximal minors have a trivial greatest common divisor. The complement of the open set of such matrices is a union of algebraic sets. The statement above shows that the union of the loci whose general elements are $ N_A, \dots, N_D$, respectively, is connected.
\end{remark}

In the next section, where we study the geometry of the degeneration loci of the above matrices, more evidence for the fact that we have four irreducible algebraic sets, one for each matrix $ N_A, \dots, N_D$ will emerge.

\subsection{Matrices with maximal minors with a common factor of degree $2$}
\label{deg2}

In this subsection, we classify the $ 4 \times 3 $ matrices whose maximal minors have a greatest common
divisor of degree $ 2$.

Let $N$ be a $4 \times 3$ matrix whose maximal minors have a greatest common divisor $q$ of degree $2.$ Let $D_i=q l_i$ for suitable linear forms
$l_i$, $i=1, \dots,4,$ be the minor, with its proper sign, obtained by removing the $i$-th row from $N.$
As $$q (l_1 , l_2 , l_3 , l_4 )N=0,$$ the matrix $N$ is a submodule of the syzygy module
of $(l_1 , l_2 , l_3 , l_4 ).$ The linear forms $l_i$ can span a subspace of dimension $4, 3$, or $2$. In the first case, they are linearly independent, in the second, we assume $l_1+z_2l_2+z_3l_3+z_4l_4=0$ for suitable scalars $z_j$, while in the last case we assume $l_1+z_{31}l_3+z_{41}l_4 = l_2+z_{32}l_3+z_{42}l_4 = 0$ for suitable scalars $z_{ij}.$ Then the syzygy modules, respectively, are generated by the columns of the following matrices:
\begin{equation}
\label{leesse}
 \scriptstyle{S_1=}\left( \begin{array}{cccccc}
\scriptstyle{0} & \scriptstyle{-l_3} & \scriptstyle{l_2} & \scriptstyle{0} & \scriptstyle{0} & \scriptstyle{l_4} \\
\scriptstyle{l_3} & \scriptstyle{0} & \scriptstyle{-l_1} & \scriptstyle{0} & \scriptstyle{l_4} & \scriptstyle{0} \\
\scriptstyle{-l_2} & \scriptstyle{l_1} & \scriptstyle{0} & \scriptstyle{l_4} & \scriptstyle{0} & \scriptstyle{0} \\
\scriptstyle{0} & \scriptstyle{0} & \scriptstyle{0} & \scriptstyle{-l_3} & \scriptstyle{-l_2} & \scriptstyle{-l_1}
\end{array} \right), \scriptstyle{S_2=} \left( \begin{array}{cccc}
\scriptstyle{1} & \scriptstyle{0} & \scriptstyle{0} & \scriptstyle{0} \\
\scriptstyle{z_2} & \scriptstyle{0} & \scriptstyle{-l_4} & \scriptstyle{l_3} \\
\scriptstyle{z_3} & \scriptstyle{l_4} & \scriptstyle{0} & \scriptstyle{-l_2} \\
\scriptstyle{z_4} & \scriptstyle{-l_3} & \scriptstyle{l_2} & \scriptstyle{0}
\end{array} \right),
\end{equation}
\begin{equation*} \scriptstyle{S_3=} \left( \begin{array}{ccc}
\scriptstyle{1} & \scriptstyle{0} & \scriptstyle{0} \\
\scriptstyle{0} & \scriptstyle{1} & \scriptstyle{0} \\
\scriptstyle{z_{31}} & \scriptstyle{z_{32}} & \scriptstyle{-l_4} \\
\scriptstyle{z_{41}} & \scriptstyle{z_{42}} & \scriptstyle{l_3}
\end{array} \right).
\end{equation*}
Moreover, in the first case there exists a maximal rank $6 \times 3$ matrix $X_1,$ with scalar entries, such that $N=S_1 X_1$; in the second case there exists a maximal rank $4 \times 3$ matrix $X_2,$ with linear
forms on the first row and scalars elsewhere, such
that $N=S_2 X_2$; in the last case there exists a generically maximal rank $3 \times 3$ matrix $X_3,$ with linear forms
on the first two rows and scalars on the last row, such that $N=S_3 X_3.$

The discussion above can now be summarized in the following result:
\begin{theorem} \label{4x3-deg2}
Let $ N $ be a $ 4 \times 3 $ matrix of linear forms that drops rank on a degree $ 2 $ hypersurface $ Q$, and let $ L $
be the linear subspace where $ N $ drops rank, residual of $ Q $. Then, $ N = S_i X_i $ where $ i = 5 - \codim(L),$ and $S_i, X_i$ are as described above.
\end{theorem}


\section{Geometry of the degeneration loci}
\label{DegLoc}

We now turn our attention to degeneration loci of matrices, i.e. we study the loci where matrices classified in Section \ref{34mat} drop rank, under some mild generality assumptions and in the dimensional context of interest for our computer vision goals. In this section, we assume  $ R = \mathbb{C}[x_0, \dots, x_4] $ so that  degeneration loci
are algebraic schemes in $ \mathbb{P}^4 = \mbox{Proj}(R)$.

We begin our analysis with matrices $N_A, N_B, N_C, N_D,$ as classified in Theorem \ref{4x3-deg1}.

\begin{proposition} \label{deg-loc-4x3-1}
The degeneration locus of a general $ N_A $ is the union of a hyperplane $H_A$ and a minimal surface $S_A$ of degree $ 3 $ in $ \mathbb{P}^4$.
\end{proposition}

\begin{proof} The defining ideal of the degeneracy locus of such a matrix is
$$ I = \left( n_{43} \left( n_{11} n_{22} - n_{21} n_{12} \right), n_{43} \left( n_{31} n_{12} - n_{11} n_{32} \right), n_{43} \left( n_{21} n_{32} - n_{31} n_{22} \right) \right).$$
Of course, the vanishing locus of $ n_{43} $ is a hyperplane in $ \mathbb{P}^4$.
Furthermore, $ I : \langle n_{43} \rangle $ is the ideal generated by the $ 2 \times 2 $ minors
of $$ \left( \begin{array}{cc} n_{11} & n_{12} \\ n_{21} & n_{22} \\ n_{31} & n_{32} \end{array} \right).$$
It is known that such minors define a minimal surface of degree $ 3 $ in $ \mathbb{P}^4,$ that is to say,
a rational normal scroll in the general case, as well as a cone over a rational normal curve contained in a suitable
hyperplane, or a degeneration of one of the above surfaces.
\end{proof}

\begin{proposition} \label{deg-loc-4x3-2}
The degeneration locus of a general $ N_B $ in $ \mathbb{P}^4 $ is a
hyperplane $ H_B$, and the union of a dimension $ 2 $ linear space $ L_B $ and a
twisted cubic curve $ C_B \subseteq H_B$. Moreover, $ H_B \cap L_B $ is a
line $r_B$ that meets $ C_B $ in two points.
\end{proposition}

\begin{proof} In this case, $ n_{31} $ and $ n_{41} $ are linearly independent, and $ n_{13} $ is the common factor to all the maximal minors of $ N_B$. Let $ H_B $ be the
hyperplane defined by $ n_{13}=0 .$ If $ I $ is the ideal generated
by the maximal minors of $ N_B $ and $ J = I : \langle n_{13} \rangle$, we have
$$ J = \langle n_{13} n_{31}, n_{13} n_{41}, q_3 = n_{31} n_{42} - n_{41} n_{32}, q_2 = n_{31} n_{43} - n_{41} n_{33} \rangle.$$
The degree $ 2 $ forms $ q_3 $ and $ q_2 $ are two of the three maximal minors of the $ 2 \times 3 $
matrix consisting of the last two rows of $ N_B$. For completeness, let $ q_1 = n_{32} n_{43} - n_{42} n_{33}$
 be the last maximal minor of such a matrix.
One can then easily obtain the minimal free resolution of
$ J :$  $$ 0 \to R(-5) \stackrel{M_2}{\lra} \begin{array}{c} R^3(-3) \\ \oplus
\\ R(-4) \end{array} \stackrel{M_1}{\lra} R^4(-2) \to J \to 0, $$ where
$$ M_1 = \left( \begin{array}{cccc} n_{41} & n_{42} & n_{43} & 0
\\ -n_{31} & -n_{32} & -n_{33} & 0 \\ 0 & -n_{13} & 0 & -q_2 \\ 0 & 0 & -n_{13} & q_3 \end{array} \right), \quad
M_2 = \left( \begin{array}{c} q_1 \\ -q_2 \\ q_3 \\ n_{13} \end{array} \right).$$ From the resolution,
the Hilbert polynomial of the scheme defined by $ J $ is $ p(t) = \frac 12 t^2 + \frac 92 t $
and so the top dimensional part of such a scheme is a linear space $ L_B $ of dimension $ 2.$
As $ J \subseteq \langle n_{31}, n_{41} \rangle$, then the defining ideal of $ L_B $ is $ \langle n_{31}, n_{41} \rangle$.
The residual scheme is defined by the ideal $ J : \langle n_{31}, n_{41} \rangle$. Let
$ K = \langle n_{13}, q_1, q_2, q_3 \rangle,$ then we claim  $ J : \langle n_{31}, n_{41} \rangle = K.$

To show $ K \subseteq J : \langle n_{31}, n_{41} \rangle$ one needs only to verify that
$ n_{31} q_1 $ and $ n_{41} q_1 $ are in $ J.$ This follows from $ n_{i1} q_1 - n_{i2} q_2 + n_{i3} q_3 = 0 $ for $ i=3,4.$

To show $ K \supseteq J : \langle n_{31}, n_{41} \rangle$, let $ f \in J : \langle n_{31}, n_{41} \rangle ,$ and show $ f \in K$. From the assumption, both
 $ f n_{31} \in J $ and $ f n_{41} \in J$,
 thus $ n_{31} f = \alpha_1 n_{13} n_{31} + \alpha_2 n_{13} n_{41} + \alpha_3 q_3 + \alpha_4 q_2$ and $ n_{41} f = \beta_1 n_{13} n_{31} + \beta_2 n_{13} n_{41} + \beta_3 q_3 + \beta_4 q_2.$ Then
 $$ ( n_{41} \alpha_1 - n_{31} \beta_1) n_{13} n_{31} + ( n_{41} \alpha_2 - n_{31} \beta_2
  )n_{13} n_{41} + ( n_{41} \alpha_3 - n_{31} \beta_3 ) q_3 + ( n_{41} \alpha_4 - n_{31} \beta_4
 ) q_2 = 0.$$
 Then, there exists a matrix $ C $ of type $ 4 \times 1 $ such that
 $$ \left( \begin{array}{c} n_{41} \alpha_1 - n_{31} \beta_1 \\ n_{41} \alpha_2 - n_{31} \beta_2 \\ n_{41} \alpha_3 - n_{31}
 \beta_3 \\ n_{41} \alpha_4 - n_{31} \beta_4 \end{array} \right) = M_1 C = M_1
 \left( \begin{array}{c} c_1 \\ c_2 \\ c_3 \\ c_4 \end{array} \right).$$
 Standard computations give $ f = n_{13} \left( a + n_{31} d_1 + n_{41} d_2 \right) - b q_1 + d_4 q_2 + d_3 q_3$, where  $$ C = \left( \begin{array}{c} a \\ 0 \\ 0 \\ b \end{array} \right) + c M_2,
 \quad \left( \begin{array}{c} \alpha_1 \\ \alpha_2 \\ \alpha_3 \\ \alpha_4 \end{array} \right) =
 \left( \begin{array}{c} a \\ 0 \\ n_{33} b \\ -n_{32} b \end{array} \right) + n_{31}
 \left( \begin{array}{c} d_1 \\ d_2 \\ d_3 \\ d_4 \end{array} \right),$$ and
 $$ \left( \begin{array}{c} \beta_1 \\ \beta_2 \\ \beta_3 \\ \beta_4 \end{array} \right) =
 \left( \begin{array}{c} 0 \\ a \\ n_{43} b \\ -n_{42} b \end{array} \right) + n_{41}
 \left( \begin{array}{c} d_1 \\ d_2 \\ d_3 \\ d_4 \end{array} \right).$$
 Hence $f \in K,$ and thus $ K = J : \langle n_{31}, n_{41} \rangle$, as claimed.
 $ K $ is the defining ideal of a twisted cubic curve
 $ C_B $ contained in the hyperplane $ H_B $ because $ q_1, q_2, q_3 $ are the $ 2 \times 2 $ minors of a $ 2 \times 3 $
 general matrix.

Lastly, $r_B =  H_B \cap L_B $ is defined by $ n_{13}, n_{31}, n_{41}, $ hence $r_B$ is a line if the linear forms are linearly independent. Moreover, $ r_B \cap C_B $ is defined by $ n_{13}, n_{31}, n_{41}, q_1 $ and so it has degree $ 2$.
Hence, $ r_B $ is a secant line to $ C_B.$
\end{proof}

\begin{proposition} \label{deg-loc-4x3-4}
The degeneration locus of a general $ N_C $ in $ \mathbb{P}^4 $ is the union of a hyperplane
$ H_C$ and two dimension $ 2 $ linear spaces $ L_{C1}, L_{C2}$ that meet at a point $p \in H_C$.
\end{proposition}

\begin{proof} The maximal minors of $ N_C $ have the linear factor $ \alpha n_{31} + \beta n_{22} $ in common, which defines the hyperplane $ H_C.$

The residual locus is defined by the ideal
$$ I = \langle n_{22} n_{31}, n_{22} n_{41}, n_{31} n_{42}, n_{41} n_{42} \rangle.$$
In the general case, we have $ I = \langle n_{22}, n_{42} \rangle \cap \langle n_{31}, n_{41} \rangle.$ Thus the degeneracy locus is equal to
$ H_C \cup L_{C1} \cup L_{C2},$ where $ L_{C1} $ and $L_{C2}$ are 2-dimensional linear spaces defined by $ \langle n_{31}, n_{41} \rangle $ and  $ \langle n_{22}, n_{42} \rangle$, respectively. Moreover, $ L_{C1} \cap L_{C2} $ is defined by $ \langle n_{22}, n_{31}, n_{41}, n_{42} \rangle$ and so
it is a point in $ H_C$.
\end{proof}

\begin{proposition} \label{deg-loc-4x3-6}
The degeneration locus of a general $ N_D $ in $ \mathbb{P}^4 $ is a hyperplane $ H_D$, and the union of a quadric surface $ Q_D $ and a line $ r_D $ in
$ H_D $. Moreover, $ Q_D \cap r_D $ is a point.
\end{proposition}

\begin{proof} The degeneracy locus of $ N_D $ is defined by the following ideal:
$$ I = \langle n_{13} n_{31}^2, n_{13} n_{31} n_{41}, n_{13} n_{31} n_{42}, n_{31}
( n_{31} n_{43} - n_{42} n_{23} - n_{41} n_{33}) \rangle.$$ It is evident that
$ I \subseteq \langle n_{31} \rangle$, and so the top dimensional part of the degeneracy locus is the hyperplane $ H_D$ defined by $n_{31}.$
The residual is defined by the ideal \begin{equation*} \begin{split} J = I : \langle n_{31} \rangle =& \langle n_{13} n_{31}, n_{13} n_{41}, n_{13} n_{42},
n_{31} n_{43} - n_{42} n_{23} - n_{41} n_{33} \rangle = \\
=& \langle n_{13}, n_{31} n_{43} - n_{42} n_{23} - n_{41} n_{33} \rangle \cap \langle n_{31}, n_{41}, n_{42} \rangle. \end{split} \end{equation*}
Let $ Q_D $ be the quadric surface defined by
$ \langle n_{13}, n_{31} n_{43} - n_{42} n_{23} - n_{41} n_{33} \rangle, $ and let $ r_D $ be the line defined by
$ \langle n_{31}, n_{41}, n_{42} \rangle$, which is contained in $ H_D$. Then the degeneracy locus of $N_D$ is as in the statement.

The intersection $r_D \cap Q_D $ is defined by the ideal $ \langle n_{13}, n_{31}, n_{41}, n_{42} \rangle $ which gives a point,
as claimed.
\end{proof}

Now, we consider the matrices in Theorem \ref{4x3-deg2}.
\begin{proposition} \label{deg-loc-4x3-7}
Let $ N $ be a general $ 4 \times 3 $ matrix of linear forms that drops rank on a degree $ 2 $ hypersurface,
as in Theorem \ref{4x3-deg2}.
\begin{enumerate}
\item If $ N = S_1 X_1$, then its degeneration locus in $ \mathbb{P}^4 $ is the union of a cone $ Q $ over a smooth quadric surface in $ \mathbb{P}^3$, and its vertex.
\item If $ N = S_2 X_2$, then its degeneration locus in $ \mathbb{P}^4 $ is a smooth quadric hypersurface $ Q $
and a line $ r $ in $ Q$.
\item If $ N = S_3 X_3$, then its degeneration locus in $ \mathbb{P}^4 $ is the union of a smooth quadric hypersurface $ Q $
and a $2$--dimensional linear space.
\end{enumerate}
\end{proposition}

\begin{proof} Let us consider first the case in which $ N = S_1 X_1.$ The maximal minors of $N$, with appropriate sign, are given by $ q \ell_1, \dots, q \ell_4, $ where $ \ell_1, \dots, \ell_4 $ are linearly independent linear forms and $ q $ is a quadratic form. Let $ \left(i,j,k\right) $ denote the determinant of the submatrix of $X_1$ consisting
 of the $ i$--th, $j$--th and $ k$--th row. Consider the symmetric matrix:

$$ D = \frac 12 \left( \begin{array}{cccc} \scriptstyle{-2 \left(2,3,6\right)} & \scriptstyle{-\left(2,3,5\right)+\left(1,3,6\right)} & \scriptstyle{-\left(2,3,4\right)-\left(1,2,3\right)} & \scriptstyle{\left(3,4,6\right)+\left(2,5,6\right)} \\
\scriptstyle{-\left(2,3,5\right)+\left(1,3,6\right)} & \scriptstyle{2 \left(1,3,5\right)} &
\scriptstyle{\left(1,3,4\right)-\left(1,2,5\right)} & \scriptstyle{\left(3,4,5\right)-\left(1,5,6\right)} \\
\scriptstyle{-\left(2,3,4\right)-\left(1,2,3\right)} & \scriptstyle{\left(1,3,4\right)-\left(1,2,5\right)} &
\scriptstyle{-2\left(1,2,4\right)} & \scriptstyle{-\left(2,4,5\right)-\left(1,4,6\right)} \\
\scriptstyle{\left(3,4,6\right)+\left(2,5,6\right)} & \scriptstyle{\left(3,4,5\right)-\left(1,5,6\right)} &
\scriptstyle{-\left(2,4,5\right)-\left(1,4,6\right)} & \scriptstyle{2 \left(4,5,6\right)} \end{array} \right) .$$
Standard direct computations show that $$ q = \left( \begin{array}{ccc} \ell_1 & \dots & \ell_4 \end{array} \right)
D \left( \begin{array}{c} \ell_1 \\ \vdots \\ \ell_4 \end{array} \right) .$$
Notice that  $ q $ defines a cone over a quadric of $ \mathbb{P}^3 $ because its
 equation depends only on $ 4 $ linearly independent linear forms. The vertex is defined by
 $ \ell_1 = \dots = \ell_4 = 0$. The smoothness of the quadric can be obtained by computing the determinant of $ D $ with the help of any software for symbolic computations.

In the second case, the maximal minors of $N$, with appropriate sign, are
still given by $ q \ell_1, \dots, q \ell_4 $ but now $ \ell_1 = -z_2
\ell_2 - z_3 \ell_3 - z_4 \ell_4.$ Moreover one has:
 $$ q = \ell_2 \left(1,3,4\right) - \ell_3
\left(1,2,4\right) + \ell_4 \left(1,2,3\right).$$ One can verify on a random numerical example that the general $ q $ defines a smooth quadratic
hypersurface. Apart from the vanishing locus of $ q$, the
degeneracy locus of $N$ is given by the three linearly independent linear forms $ \left( \ell_2, \ell_3, \ell_4\right) $ defining a line $r.$ From the previous description of $ q$, we
get that $ r\subseteq Q.$

In the remaining case, the maximal minors with sign of $ N $ are once again given by
$ q \ell_1, \dots, q \ell_4$, but now $ \ell_1 = -z_{31}
\ell_3 - z_{41} \ell_4 $ and $ \ell_2 = -z_{32} \ell_3 - z_{42}
\ell_4$. Then, the degeneracy locus is defined by $$ \langle q
\rangle \cap \langle \ell_3, \ell_4 \rangle.$$ The quadratic form $
q $ is equal to $ \det(X_3) $ and so it defines a smooth quadratic
hypersurface in $ \mathbb{P}^4$. The second ideal defines a linear
space of dimension $ 2 $ because the two linear forms are linearly
independent.
\end{proof}

\section{Multiview Geometry for projections from $ \mathbb{P}^4 $
to $ \mathbb{P}^2$} \label{prelimCV}
A classical problem in Computer Vision is {\it reconstruction}:
 given multiple images of an unknown scene,
taken from unknown cameras, try to reconstruct the positions of the
cameras and of the scene points. It is not difficult to see that
sufficiently many images, and sufficiently many sets of
corresponding points in the given images, chosen as images of the same set of test points in space, should allow for a
successful projective reconstruction. As mentioned in the introduction, similar problems have been investigated in a more general framework of projections from $\Pin{k}$ to $\Pin{m},$ and we are focusing on the case of three projections from $\Pin{4}$ to $\Pin{2}.$

As above, $\Pin{k}$ denotes the $k-$dimensional complex projective space. Once a projective frame is
chosen for $\Pin{k},$ coordinate vectors $\mathbf{X}$ of points in
$\Pin{k}$ are written as columns, thus $\mathbf{X}^T
=(X_1,X_2,...,X_{k+1}).$ A {\it camera} $P$ is
a linear projection from $\Pin{4}$ onto $\Pin{2},$  from a line
$C_P,$ called {\it center} of projection. The target space
$\Pin{2}$ is called a {\it view}. A {\it scene}  is a set of points
$ \nbXi \in \Pin{4}.$

The camera $P$ is identified with a $3 \times 5 $ matrix of
maximal rank, defined up to a multiplicative constant.  Hence
$C_P$ comes out to be the right annihilator of $P.$

If $\mathbf{X}$ is a point in $\Pin{4},$ we denote its image in
the projection equivalently as $P(\mathbf{X})$ or $P \cdot
\mathbf{X}.$

Given different projections with skew centers,
$P_i:\Pin{4}\setminus C_{P_i} \to \Pin{2}$,  $i=1 \dots m,$ the
images $E_{r,s} = P_r(C_s)$ ($r,s =1 \dots m, r \neq s$)  are
lines of the view spaces, usually called {\it epipoles}.

Proper linear subspaces (points or lines), $L_i, i=1 \dots m,$ of
different views  are said to be \textit{corresponding} if there
exists at least one point $\mathbf{X} \in \Pin{4}$ such that
$P_i(\mathbf{X})\in L_i $ for all $i=1 \dots m.$

Hartley and Schaffalitzky,  \cite{Hart-Schaf}, constructed a
set of multiview tensors, called {\it Grassmann tensors}, encoding
the relations between sets of corresponding subspaces in the general settings of multiple projections from $\Pin{k}$ to $\Pin{h_i}.$ Such tensors, which generalize the notion of {\it fundamental matrix}, are the key ingredient of the process of reconstruction. Indeed, once such a tensor is obtained, one can proceed to reconstruct cameras and scene points.
For our purposes, here we recall the definition of such tensors in the case of three projections from $\Pin{4}$ to $\Pin{2}.$

Consider three projections $P_j:\Pin{4}\setminus{C_{P_j}} \to
\Pin{2},$ $j = 1,2,3,$ with centers $ C_{P_1}, C_{P_2}, C_{P_3} $
in general position. A {\em profile} is a
partition $(\alpha_1, \alpha_2, \alpha_3)$ of $k+1 = 5,$ i.e. $1
\leq \alpha_j \leq 2$ for all $j,$ and $\sum\alpha_j = 5.$ The only possible profiles are: $(\alpha_1,
\alpha_2,\alpha_3)=(2,2,1),(2,1,2)$ or $(1,2,2).$

Let $\{L_1, L_2, L_3\} $ be three general linear subspaces of
$\Pin{2}$ of codimension $\alpha_1, \alpha_2, \alpha_3$,
respectively. Let $S_j$ be the maximal rank $3 \times (3-
\alpha_j)$-matrix whose columns are a basis for $L_j, j=1,2,3 $.
By definition, if all the $L_j$s are corresponding subspaces
there exists a point $\mathbf{X} \in \Pin{k}$ such that
$P_j(\mathbf{X})\in L_j$ for $j=1,2, 3.$ In other words there
exist three vectors $\mathbf{v_j} \in \mathbb{C}^{3-\alpha_j}$ $j
= 1,2,3$ such that:
\begin{equation}
\label{grasssystem}
\left(%
\begin{array}{cccc}
  S_1 & 0 & 0 & P_1 \\
  0 & S_2 &0 & P_2 \\
  0 & 0 & S_3 & P_3 \\
\end{array}%
\right)\cdot
\left(%
\begin{array}{c}
  \mathbf{v_1} \\
  \mathbf{v_2} \\
  \mathbf{v_3} \\
  \mathbf{X} \\
\end{array}%
\right)=\left(%
\begin{array}{c}
  0 \\
  0 \\
  0 \\
\end{array}%
\right).
\end{equation}

The existence of a non--trivial solution
$\{\mathbf{v_1},\mathbf{v_2},\mathbf{v_3},\mathbf{X} \}$ of the
system (\ref{grasssystem}) implies that the $9 \times 9$
coefficient matrix has determinant zero. This determinant can be
thought of as a tri-linear form (tensor) in the Pl\"{u}cker
coordinates of the spaces $L_j$. This tensor is the {\it Grassmann
tensor}.

We explicitly construct such a tensor for the
profile $(2,2,1),$ others being similar. In the case of the chosen profile, $L_1,L_2$ are points
and $L_3$ is a line. We denote by ${\mathbf x}=(x_1,x_2,x_3)^T,
{\mathbf y}=(y_1,y_2,$ $y_3)^T$ the homogeneous coordinates of $L_1$
and $L_2$, respectively, and by ${\mathbf z}=(z_1,z_2,z_3)^T$ and
${\mathbf w}=(w_1,w_2,$ $w_3)^T$ the homogeneous coordinates of two
points of $L_3.$ Then the matrix of the coefficients of
the linear system (\ref{grasssystem}) becomes:
\begin{equation*} \begin{split}
& T^{P_1,P_2,P_3}_{L_1,L_2,L_3}=\left(%
\begin{array}{ccccc}
  {\mathbf x} & 0 & 0 & 0 & P_1 \\
  0 & {\mathbf y} & 0 &0 & P_2 \\
  0 & 0 & {\mathbf z} & {\mathbf w}& P_3 \\
\end{array}%
\right)\end{split}
\end{equation*}

 \medskip
If $L_1,L_2$ and $L_3$ are corresponding spaces then the linear
system

\begin{equation}
\label{Msystem}
T^{P_1,P_2,P_3}_{L_1,L_2,L_3} \ \left(%
\begin{array}{c}
  \lambda \\
  \mu \\
  \alpha \\
  \beta \\
  \mathbf{X} \\
\end{array}
\right)=\mathbf{0}
\end{equation}

\noindent has a non trivial solution, and so
$\det(T^{P_1,P_2,P_3}_{L_1,L_2,L_3})=0.$

\noindent The converse is true for general $L_1, L_2$ and $L_3$
since we are looking for a non trivial solution of \brref{Msystem}
in which $\mathbf{X}$ is a point of  $\Pin{4}$ and hence
$\mathbf{X} \neq \mathbf{0},$ and $\mathbf{X} \notin C_1 \cup C_2
\cup C_3.$

In particular this happens if $L_1 \notin E_{1,2} \cup E_{1,3}$
and $L_2 \notin E_{2,1} \cup E_{2,3}$. Under this hypothesis, if
$\det(T^{P_1,P_2,P_3}_{L_1,L_2,L_3})=0,$ then $L_1, L_2$ and $L_3$
are corresponding spaces as the linear system \brref{Msystem} has a non
trivial solution with $\mathbf{X}$ as required. Indeed, if
$\mathbf{X} = \mathbf{0}$ were part of the solution, either $ L_1 $ or $ L_2 $ are not proper points, or $ L_3 $ is not
a line. Moreover $\mathbf{X} \notin C_1 \cup C_2 \cup C_3, $ from
our initial hypothesis on the epipoles.
 \medskip

In conclusion, for the chosen profile $(2,2,1),$ one sees that
$\det(T^{P_1,P_2,P_3}_{L_1,L_2,L_3})=0$ is indeed the tri--linear
constraint between the coordinates $\mathbf{x}$ and $\mathbf{y}$
of points in the first and second view and the Pl\"{u}cker (i.e. dual)
coordinates of lines $<\mathbf{z},\mathbf{w}>$ in the third
view, encoding the fact that $L_1,L_2,L_3$ are corresponding spaces.

More explicitly, denoting by $p_1,p_2,p_3$ the dual Pl\"{u}cker
coordinates of $L_3$, with an iterated application of the
generalized Laplace expansion, one gets:
\begin{equation}
\label{reltens} det(T^{P_1,P_2,P_3}_{L_1,L_2,L_3})=
\sum_{i,j,k}T_{i,j,k}x_i y_j p_k
\end{equation}
where $i,j,k=1,2,3$ and where the entries $T_{i,j,k}$ of the
tensor are given by:
\begin{equation}
\label{Tijk} T_{i,j,k}=(-1)^{(i+j+k+1)} \det
\begin{bmatrix}
  P_{1_{\widehat{i}}} \\
  P_{2_{\widehat{j}}} \\
  P_3(4-k) \\
\end{bmatrix} \end{equation}

\noindent where, as above,  $ P_{1_{\widehat{i}}}$ and
  $P_{2_{\widehat{j}}}$ are
obtained by $P_1$ and $P_2$ respectively deleting rows $i$ and
$j$, and $P_3(4-k)$ denotes the row $4-k$ of the matrix $P_3$.

\section{Critical loci: general set up}
\label{critlocsetup} As discussed in the previous section,
sufficiently many views and sufficiently many sets of
corresponding points in the given views, should allow for a
successful projective reconstruction. This is generally true, but
even in the classical set up of two projections from $\Pin{3}$ to
$\Pin{2}$ one can have non projectively equivalent pairs of sets
of scene points and of cameras that produce the same images in the
view planes, from a projective point of view, thus preventing
reconstruction. Such configurations and the loci they describe are
referred to as {\it critical}. In \cite{LAIA} critical loci for
projective reconstruction of camera centers and scene points from
multiple views for projections from $\Pin{k}$ to $\Pin{2}$ have
been introduced and studied. Here we shortly recall the basic
definitions in the case of interest, i.e. three
views from $\Pin{4}$ to $\Pin{2}.$

 A set of points $\bxj,$ $j=1,\dots,N,$ $N\ge 7,$
in $\Pin{4}$ is said to be a \textit{critical configuration for
projective reconstruction from $3$-views} if there exist a
non-projectively equivalent set of $N$ points
$\{\mathbf{Y}_j\}\subset \Pin{4}$ and two collections of $3 \times
5$ full-rank projection matrices $P_i$ and $Q_i,$ $i=1,2,3,$ such
that, for all $i$ and $j$, $P_i \cdot \mathbf{X}_j = Q_i \cdot
\mathbf{Y}_j$, up to homography in the image planes. Critical configurations arising from a given pair of triples of projections $P_i$ and $ Q_i$ define a scheme called {\it critical locus}.

As shown in \cite{bntMEGA}, the generators of the
ideal of the critical locus $\mathcal{X}$ can be obtained directly
making use of the Grassmmann tensor built with the $Q_i$s, $det(T^{Q_1,Q_2,Q_3}_{L_1,L_2,L_3}),$ introduced in section
\ref{prelimCV}. The idea is that if $\mathbf{
X}$ is in the critical locus, the points $P_1(\mathbf{X}),$ $P_2(\mathbf{X})$ and $P_3(\mathbf{X})$ are corresponding not only for the $P_i$s but also for the $Q_i$s. This implies that
$\mathbf{X}$ is in the critical locus $\mathcal{X}$  if and only if all the
maximal minors of the following matrix vanish:

\begin{equation}
\label{emme}
M = \left(%
\begin{array}{cccc}
  P_1(\mathbf{X}) & 0 & 0 &  Q_1 \\
  0 & P_2(\mathbf{X}) & 0 &  Q_2 \\
  0 & 0 & P_3(\mathbf{X}) &  Q_3 \\
\end{array}%
\right).
\end{equation}

\subsection{The general case of critical loci}

As we have seen in the previous section, the critical locus
$\mathcal{X}$ turns out to be a determinantal variety, associated to the matrix $M$ as above. In \cite{bntMEGA} the authors study $\mathcal{X}$  when the matrix $M$ satisfies two key hypothesis:
\begin{itemize}
\item[i)] $M$ does not drop rank in codimension 1;
\item[ii)] the last five columns of $M$ are linearly independent.
\end{itemize}
Hypothesis ${\rm ii)}$ has a computer vision interpetation. If these columns were linearly dependent, the centers of projection of the matrices $Q_i$ would intersect in at
least one point, in which case the entire ambient space would be critical.
Hypothesis ${\rm i)}$ has important algebro-geometric consequences. Indeed it implies that  Hilbert--Burch Theorem can be applied to the ideal of $\mathcal{X}$, enabling the computation of its generators and their first syzygy module, giving a
scheme--theoretical description of the critical locus itself. It
turns out that the ideal is minimally generated by $ 4 $ degree-$
3 $ forms that are the maximal minors of a $ 4 \times 3 $ matrix
$N_{\mathcal{X}}$ with linear entries. Moreover, the entries of
each column of $N_{\mathcal{X}}$ define four hyperplanes meeting
in a line. These three lines are indeed the centers of projection
of the $P_i$. Hence it follows that the critical locus is a
determinantal variety of codimension $ 2 $ and degree $ 6 $ in $
\mathbb{P}^4$, and so it belongs to the irreducible component of
the Hilbert scheme containing the Bordiga surfaces. We
recall that a Bordiga surface $ S $ is the blow--up of $
\mathbb{P}^2 $ at ten general points, embedded in $ \mathbb{P}^4$
via the complete linear system of the quartics through the $ 10 $
points.
\subsection{The degenerate cases of critical loci}

In this section we study the same problem analyzed in
\cite{bntMEGA} and recalled above, assuming now that the maximal
minors of $M$ have a non trivial common factor, and hence the
hypothesis of the Hilbert--Burch Theorem are not
satisfied.  Obviously we still assume hypothesis ${\rm ii)}$ as above, i.e. that $ \mbox{rank}\left(%
\begin{array}{c}
   Q_1 \\
   Q_2 \\
   Q_3 \\
\end{array}%
\right)=5.$

Following \cite{bntMEGA}, we now construct a $ 4 \times 3$ matrix
$N_{\mathcal{X}},$ with linear entries, whose maximal minors
define the same ideal as those of the matrix $M$.

Up to elementary row operations,  $ M $ can be written as the
following block matrix: $$ M = \left(
\begin{array}{cc} A & B
\\ C & D \end{array} \right), $$ where $ A $ is of type $ 4 \times 3$ and
 $ D $ is of type $5 \times 5,$ invertible. A new series of elementary
operations on rows and columns can then turn $M$ into the block matrix:
$$
 \left( \begin{array}{cc} N_{\mathcal{X}} & 0 \\ 0 & I_5 \end{array} \right)
 $$
where $ I_5 $ is the $ 5 \times 5 $ identity matrix, $ 0 $ are
null matrices of suitable type, and
\begin{equation}
\label{NfromM} N_{\mathcal{X}} = A - B
D^{-1} C.
\end{equation}
 The matrix $N_{\mathcal{X}}$ is a $ 4 \times 3$ matrix
of linear forms, the minors of which have a non trivial common
factor and hence it is one of the matrices classified in Section
\ref{34mat} whose degeneracy loci are studied in Section \ref{DegLoc}.
\begin{remark} \label{rankcolumns}
Notice that,in view
of \brref{emme} and \brref{NfromM}, the $i$-th column of
$N_{\mathcal{X}}$ contains linear
forms that vanish on the center of
projection $C_{P_i}.$ Therefore each column of $N$ can contain at most three linearly independent linear forms.
\end{remark}

Conversely, following \cite{bntMEGA}, for a given a matrix $N_{\mathcal{X}}$ as above,  it is possible to
recover projection matrices $ P_1, P_2, P_3, Q_1,Q_2, Q_3 $ whose critical locus is given by the minors of
$N_{\mathcal{X}}.$

The above remark gives a necessary condition for one of the matrices classified in
section \ref{34mat} to appear in the setting of multiview
geometry. This condition is now checked for all cases
appearing in Section \ref{34mat}, and the critical locus is
described, following Section \ref{DegLoc} for admissible cases.

\begin{theorem}
\label{crtiloctheo} With notations as above, let $ P_1, P_2, P_3,$
and $Q_1,Q_2, Q_3 $ be two sets of projections from $\Pin{4}$ to
$\Pin{2}$ whose critical locus is defined by a matrix
$N_{\mathcal{X}}$ whose maximal minors have a non trivial common
factor. Then the critical locus is as in one of the following
cases:

\begin{itemize}
\item [i)] the union of a hyperplane $H_A$ and a minimal surface
$S_A$ of degree $ 3 $ in $ \mathbb{P}^4$. The three
centers of projections $C_{P_i}$ are contained in $S_A.$

\item [ii)] a hyperplane $ H_B$, and the union of a $2$-space $
L_B $ and a twisted cubic curve $ C_B \subseteq H_B$. Moreover, $
H_B \cap L_B $ is a line $r_B$ that meets $ C_B $ in two points.
The three centers of projections $C_{{P_i}}$ are
contained in $H_B.$

\item [iii)] a hyperplane $ H_D$, and the union of a quadric
surface $ Q_D $ and a line $ r_D $ in $ H_D $. Moreover, $ Q_D
\cap r_D $ is a point and the three centers of projection $C_{P_{i}}$ are contained in a 2-space.

\item[iv)]the union of a cone $ Q $ over a smooth quadric surface in $ \mathbb{P}^3$, and its vertex. The three centers of projection $C_{P_{i}}$ are contained in $Q.$

\item[v)]a smooth quadric hypersurface $ Q $ and a line $ r
\subset Q$. The three centers of projections
$C_{P_i}$ are contained in $Q.$

\item [vi)] the union of a smooth quadric hypersurface $ Q $ and a
$2$-space $\Pi.$ The three centers of
projections $C_{P_{i}}$ are contained in $\Pi.$
\end{itemize}
\end{theorem}
\begin{proof}
Our assumptions imply that, up to elementary row and column operations, $N_{\mathcal{X}}$ is one of the $4 \times 3$ matrices
classified in Section \ref{34mat}.
Considering matrices as in Theorem \ref{4x3-deg1}, one sees that:
\begin{itemize}
    \item[A)]$\NX$ can be of the form $N_A.$ Indeed, specialization of the first two columns, and suitable elementary
    operations on the third column, bring $N_A$ to a matrix whose three columns define lines. From
    Proposition \ref{deg-loc-4x3-1}, we obtain case $ i)$ in the statement.
    \item[B)]$\NX$ can be of the form $N_B.$ Indeed, as above, suitable elementary operations on the columns, bring $N_B$
    to a matrix whose three columns define lines, all contained in the hyperplane $n_{13} = 0.$ From
    Proposition \ref{deg-loc-4x3-2} we obtain case $\rm ii)$ in the statement.
    \item[C)]$\NX$ can not be of the form $N_C.$ From our general assumptions, the third column defines a line.
    Swapping the first and third column in $N_C,$ we can see that in this
    case, in \brref{NfromM}, it is
    \begin{equation} A = \left(
    \begin{array}{ccc} \alpha n_{31} + \beta n_{22}&0 & 0 \\\alpha n_{41}& 0 & 0  \\ \beta n_{42}& 0 & 0\\ 0 & n_{42} & 0 \end{array} \right)\end{equation} and the first column of $C$ is the zero vector. Moreover, all linear forms on the remaining columns of $C$ are linearly independent. As the first row of $\NX$ contains two zeros, it follows that the first row of $BD^{-1}$ must be the zero vector, which implies that the first row of $B$ is the zero vector, which is impossible as projection matrices $Q_i$ are of maximal rank.
    \item[D)]$\NX$ can be of the form $N_D.$ Adding to the first and second column of $N_D$ suitable multiples of
    the third one, with non-zero coefficients $a$ and $b$ respectively, one can bring $N_D$ to a matrix whose three columns define
    lines, the first two of which are contained in the 2-space $n_{13} = n_{31} + b n_{23} + a n_{33} = 0,$ the first and the third of which are contained in the 2-space $n_{13} = n_{23} = 0,$ and where the second and third of which are contained in the 2-space $n_{13} = n_{33} = 0.$
    From Proposition \brref{deg-loc-4x3-6} we obtain case $ iii)$.
    \end{itemize}
    Considering matrices as in Theorem \ref{4x3-deg2}, one sees that:
    \begin{itemize}
    \item[1)]$\NX$ can be of the form $S_1 X_1.$ In this case, the $j$-th column of $ N $ is $ (N)_j = S_1 \left( \begin{array}{c} x_{1j} \\ \vdots \\ x_{6j} \end{array} \right) = Y_j L$ where $Y_j =\left( \begin{array}{cccc} 0 & x_{3j} & -x_{2j} & x_{6j} \\ -x_{3j} & 0 & x_{1j} & x_{5j} \\ x_{2j} & -x_{1j} & 0 & x_{4j} \\ -x_{6j} & -x_{5j} & -x_{4j} & 0 \end{array} \right)$ and $L = \left( \begin{array}{c} l_1 \\ l_2 \\ l_3 \\ l_4 \end{array} \right).$ As $Y_j$ is skewsymmetric, its rank is even and Remark \ref{rankcolumns} implies that $\rk{Y_j} = 2,$ hence $ x_{3j}x_{4j}+x_{2j}x_{5j}+x_{1j}x_{6j} = 0$. Recalling \brref{emme} and \brref{NfromM}, this case must come from a matrix $M$ whose rows were permuted. Without loss of generality we can assume that $ N = A - B D^{-1} C ,$ where
        \begin{equation}
        \label{A_and_C_case_iv}
        A = \begin{bmatrix} P_{11} X & 0 & 0 \\ P_{12} X & 0 & 0 \\ 0 & P_{21} X & 0 \\ 0 & 0 & P_{31} X \end{bmatrix}
         \qquad
         C = \begin{bmatrix} P_{13} X & 0 & 0 \\ 0 & P_{22} X & 0 \\ 0 & P_{23} X & 0 \\ 0 & 0 & P_{32} X \\ 0 & 0 & P_{33} X \end{bmatrix}.
        \end{equation}Let $ T = (t_{ij})=B D^{-1}.$ Then the first column of $ N $ is $$ Y_1 L = ( P_{11} X - t_{11} P_{13} X , P_{12} X - t_{21} P_{13} X , -t_{31} P_{13} X , -t_{41} P_{13} X)^T$$ where $ P_{11} X, P_{12} X,$ and $ P_{13} X $ are linearly independent. Remark \ref{rankcolumns} implies $ t_{31} = t_{41} = 0$ and therefore $ x_{11} = x_{21} = x_{41} = x_{51} = x_{61} = 0,$ and $ x_{31} \not= 0.$ Thus $ P_{11} X = t_{11} P_{13} X + x_{31} l_2,$ $P_{12} X = t_{21} P_{13} X - x_{31} l_1.$ Notice that the first projection matrix $P_1$ is therefore completely determined once the linear form $P_{13}X$ and $t_{11}, t_{21}$ are chosen.

         Proceeding in a similar fashion for the second and third column of $N,$ with somewhat laborious but standard calculations, one sees that $P_2$ and $P_3$ are also completely determined once their third rows are known, i.e. once linear forms $P_{23}X$ and $P_{33}X$ are chosen together with suitable parameters related to $t_{ij}.$

         From $T=BD^-1$ one can then retrieve the remaining matrices $Q_{i}.$ For the convenience of the reader, a numerical example of this case is found in section \ref{inst_cono_quadrico}.
         From Proposition \brref{deg-loc-4x3-7} we obtain case $iv).$
     \item[2)]$\NX$ can be of the form $S_2 X_2.$ For example,assuming all $z_i$s vanish, one can choose three generic
     linear forms as first row for $X_2$ and the $3 \times 3$ identity matrix to complete $X_2.$ From
     Proposition \ref{deg-loc-4x3-7} we obtain case $ v)$ in the statement.
     \item[3)]$\NX$ can be of the form $S_3 X_3.$ Recalling \brref{leesse}, we can assume $z_i=0$ per $i=1, \dots 4.$ Remark \ref{rankcolumns} implies that one of the entries of the last row of $X_3$ must vanish and that the linear forms of the first and second column of $X_3$ must be linearly dependent. It follows then that the entries of the first two columns of $S_3X_3$ span lines while those of the last column span a plane. $M$ can then be reconstructed as usual. Similarly to case $1)$ above, notice that the third projection matrix is not uniquely determined, as one linear form can be freely chosen. From Proposition \brref{deg-loc-4x3-7} we obtain case $vi).$

\end{itemize}
\vspace{-\belowdisplayskip}\[\]
\end{proof}

\section{Reconstruction of the trifocal tensor}
\label{rec_tens_deg}

This section is dedicated to the investigation of the actual possibility of reconstructing the trifocal tensor in the situations listed in Theorem \ref{crtiloctheo} and, when reconstruction is possible, of instability phenomena. Once the tensor is obtained, then one can further reconstruct cameras and sets of scene points, if
needed, by intersecting  projecting rays of corresponding spaces.
As previously
noted, the trifocal tensor encodes triplets of corresponding
spaces in the views.
Let $L_1=(x_1,x_2,x_3)$ and $L_2=(y_1,y_2,y_3),$ be two
points in the first two views, respectively, and let  $L_3=
<(z_1,z_2,z_3),(w_1,w_2,w_3)>$ be a line in the third view,
spanned by two given points. Then (\ref{reltens}) identifies
triplets $L_1,L_2,L_3$ of corresponding spaces. Viceversa,  given
a large enough number of triplets of corresponding spaces
$L_1,L_2,L_3$, repeated use of (\ref{reltens}) gives rise to a
linear system that, when its matrix $M_T$ has maximal rank $\rk{M_T} = 26,$   determines, up to a multiplicative
constant, the $27$ entries of the trifocal tensor $T$. When cameras are in general configurations, $\rk{M_T}$ is indeed maximal. This is reflected in  \cite{Ito} where, roughly speaking in our context, the authors study the map associating to a triplet of projection matrices the variety parameterizing the triplets of corresponding points in the views. The main result of this work is that the general fiber of the map above consists of a single orbit of $PGL(5,\mathbb{C}),$ in other words the projection matrices, and thus the tensor, are uniquely determined up to projective equivalence in $\Pin{4}.$

Although projection matrices in the cases of interest in this section are not general, one can verify that, in cases $\rm{(i,ii,iv,v)}$ the rank of the corresponding matrix $M_T$ is still maximal, regardless of the chosen profile. On the contrary, as we will see below, $\rk{M_T}$ in cases $\rm{iii}$ and $\rm{vi}$ is not maximal.

\begin{remark}
Assume now that at least two of the centers of projections, say $C_{{P_1}}$ and $C_{{P_2}},$ are
contained in a $2$-plane $\Pi,$ and hence they intersect. In this case
the trifocal tensor of profile $(2,2,1)$ is no longer uniquely determined by
constraints generated by triplets of corresponding spaces. It is
interesting to see both from a geometric and an algebraic point of
view how infinitely many tensors vanish on the same triplets of
corresponding spaces.

From the geometric point of view, one can argue as follows. Assume
that the two lines, center of projections, $C_{P_1}$ and $C_{P_2}$
intersect at a point $\mathbf{V}$. Let $\mathbf{X}$ be a scene
point and let us consider two corresponding points
$L_1=P_1(\mathbf{X})$ and $L_2=P_2(\mathbf{X})$ in the first and
second view, respectively. Let $L_3$ be any line in the third view
and consider the two planes, $\pi_1=<L_1,C_{P_1}>$ and
$\pi_2=<L_2,C_{P_2}>$ which are the projecting rays for  $L_1$ and
$L_2$. Then $\pi_1$ and $\pi_2$ meet along the line
$\lambda=<\mathbf{X},\mathbf{V}>$, which in turn intersects the
$3-$ dimensional space $<L_3,C_{P_3}>$ at a point $\mathbf{Y},$
such that $L_1=P_1(\mathbf{Y}), L_2=P_2(\mathbf{Y})$ and
$P_3(\mathbf{Y}) \in L_3.$

This shows that $L_1$,$L_2$ and $L_3$ are corresponding spaces for
any choice of $L_3$.

From the algebraic point of view, it is easy to see that the
trifocal tensor is not well defined. Indeed, one can assume that the matrices $P_i$,$ i=1,2$,
have the same first two rows. From the expression (\ref{Tijk}) of
the entries of the trifocal tensor, one sees that $T_{i,j,k}$
vanishes unless $i=1$ and $j=2$ for any $k$ or $i=2$ and $j=1$ for
any $k$.  In other words the trilinear relation in this case is
\begin{equation}\label{trireldeg2} \quad (T_{1,2,1}x_1
y_2+T_{2,1,1}x_2 y_1) p_1 + (T_{1,2,2}x_1 y_2+T_{2,1,2}x_2 y_1)
p_2+(T_{1,2,3}x_1 y_2+T_{2,1,3}x_2 y_1) p_3,\end{equation}
where $(x_1,x_2,x_3)$ and $(y_1,y_2,y_3)$ are chosen coordinates in the first two views and $(p_1,p_2,p_3)$ are Pl\"ucker coordinates of lines in the third view.
Moreover $T_{1,2,k}=-T_{2,1,k},$ for any $k$. Now, if
$(x_1,x_2,x_3) = P_1(\mathbf{X})$ and $(y_1,y_2,y_3) =
P_2(\mathbf{X})$, then $x_1 = y_1$ and $x_2 = y_2$, so that
expression \brref{trireldeg2} vanishes for any choice of
$p_1,p_2,p_3$. Moreover any trilinear relation of the form
$$\label{trireldeg2.2} \quad (ax_1
y_2-ax_2 y_1) p_1 + (bx_1 y_2-bx_2 y_1) p_2+(bx_1 y_2-bx_2 y_1)
p_3$$ vanishes as well. This means that the admissible trifocal
tensors for this profile depend on three homogeneous parameters, which shows that in these cases $\rk{M_T} = 24.$
On the contrary, assuming that the third center is not contained in $\Pi,$ a direct computation shows that, for the remaining profiles, $\rk{M_T} = 26.$
\end{remark}

\subsection{Reconstruction of the trifocal tensor - cases $iii,vi$}
In these cases all three centers of projection $C_{P_i}$ are contained in the same 2-space $\Pi.$ The above remark shows that there is no profile for which the tensor can be uniquely reconstructed.

\subsection{Reconstruction of the trifocal tensor - cases $i,ii,iv$ and $v$}

In cases $i,ii,iv$ and $v$ we consider camera centers for the $P_i$s
which are in special position in $\Pin{4},$ but do not intersect
one another. Although the critical locus, as seen above, turns out
to be reducible and of higher codimension than usual, the
reconstruction algorithm for the tensor is unaffected by the
degenerate camera configuration. As discussed above, the key
element needed to reconstruct the tensor is the existence of an
actual constraint as a result of imposing to a triplet of spaces
in the three views to be corresponding. The fact that the three
center of projections are now lines in special position does not
affect the generation of the constraints. Indeed, given any to
general points $L_1, L_2$ in the first and second
view,respectively, their projecting rays $R_1, R_2$ are $2$-spaces
in $\Pin{4}$ that meet at one point $\mathbf{X} \in \Pin{4}.$ The
projecting ray $R_3$ of a line in the third view is a $3$-space
that, in general, will not contain $\mathbf{X}.$ Hence, the
condition $\mathbf{X} \in R_3,$ needed to impose to $L_1,L_2,L_3$
to be corresponding, is an actual constraint. This is confirmed by a few numerical tests in which $\rk{M_T} = 26.$

\section{Instability results near critical loci}
\label{inst_exp}
 In this section we intend to study the instability of the
 reconstruction algorithm in the vicinity of the critical locus.
 Notice that in cases $i$ and $ii$ the reducible critical locus has a hyperplane as a component.
 When corresponding points chosen in the views are images of points which all lie in a hyperplane, whether or not this hyperplane is a component of the critical locus, $\rk{M_T}$ cannot be maximal. If it were, one could reconstruct uniquely the tensor and hence the set of projection matrices, which is clearly impossible as points lying in a hyperplane can be mapped to the same points by infinitely many different projection matrices. Therefore, we will not concern ourselves with instability phenomena starting from points lying on linear components of the critical loci. We will then only consider the cubic rational scroll $S_A$  in case $i,$ and the hyperquadric in cases $iv$ and $v.$

 As it is natural in applications, we reset the framework in an affine context, assuming that the world scene observed lies entirely within the
 affine chart given by $x_5 \neq 0$ in $\Pin{4}.$
 The experimental process to investigate instability relies on algorithms developed in \cite{tubbISVC08}, and it is described below.

\begin{itemize}
\item[1]{\it Generation of Critical Configurations} \\
Two sets of projection matrices $\{P_i\}$ and
$\{Q_i\}$, $ i = 0,1,2$ of the appropriate type, are obtained from $N_{X}$ as described in \cite{bntMEGA}. From projection matrices equations of the components of the critical locus and sets of points $\bxi$ lying on the non-linear components
are obtained with the help of Maple, using elementary geometric arguments described in the next two subsections.
\item[2]{\it Perturbation of critical configurations}\\
Points $\bxi$ are then perturbed with a $4$-dimensional
noise, normally distributed, with zero mean, and with
assigned standard deviation $\sigma$, obtaining a new
configuration $\{\nbxi^{pert}\},$ which is close to being
critical.
Configuration $\{\nbxi^{pert}\}$ is then projected via $P_1, P_2, P_3.$
The resulting images $\mathbf{x_{ij}} = P_j\nbxi^{pert}$ are
again perturbed with normally distributed 2-dimensional
noise with zero mean and standard deviation $0.01$ to
obtain $\{\mathbf{x_{ij}}^{pert}\}.$
 \item[3]{\it Reconstruction} \\
 The trifocal tensor corresponding to the true
 reconstruction, $T_P,$ is
 computed  from the $P_i$s
 using a reconstruction algorithm described in detail in \cite{tubbISVC08}, section 4.1, and implemented in Matlab.
 An estimated trifocal tensor $T$ is
 computed from $\{\mathbf{x_{ij}}^{pert}\},$ using
an algorithm described in \cite{tubbISVC08}, section 4.2 and implemented in Matlab as well.

\item[4]{\it Estimating instability}\\
As trifocal tensors are defined up to multiplication by a
non-zero constant, $T_P,$ and $T$ are normalized (using Frobenius norm) and the
space of trifocal tensors is then identified with a subset of the
quotient of the unit sphere in $\mathbb{R}^{27},$
$S^{26}/\simeq,$ where $\simeq$ denotes antipodal
identification. It is simple to account for antipodal
identification when computing distances:
 for any pair of tensors $A$ and $B,$ with unit Frobenius norm, we set
$d(A,B) = min (||A-B||, ||A+B||).$ Using this notion of distance,
we estimate whether $T$ is close to $T_P,$ or not, where ``close"
means within a hypersphere of suitable radius $\delta.$ In order
to choose a suitable $\delta$ we start from a set of random points in $\Pin{4},$ project them with the $P_i,$ perturb images using the same noise as in 2 above, reconstruct the tensor $T_{test},$ compute $d(T_P, T{test})$ (as above), repeat this procedure 1000 times, find the mean $m$ of all distances, and select $\delta$ as a suitable multiple of $m.$

 The above procedure is then
repeated 10 times for every fixed value of $\sigma.$
\end{itemize}
\subsection{Instability results in case $i$}
\label{inst_scroll}
In this case, suitable projection matrices are chosen as follows:
$$P_1 = \begin{bmatrix} 1&0 &0 &0 &0\\ 0& 1& 0& 0 &0\\0&0&1&0&0\end{bmatrix},
 P_2 = \begin{bmatrix}1&0 &0 &0 &0\\ 0&0&0&1&0\\0&0&0&0&1 \end{bmatrix},
 P_3 = \begin{bmatrix}\frac{10158}{25}&729&4050&\frac{31152}{5}&-3645\\
                      608&-\frac{13692}{25}&1900&836&\frac{13692}{5}\\
                      288&162&\frac{258}{25}&396&-810\end{bmatrix}$$
$$Q_1 = \begin{bmatrix} -1&0 & -1&0 &0\\ 0& -1& 0& -1 &0\\0&0&0&0&-1\end{bmatrix},
 Q_2 = \begin{bmatrix}0&0 &\frac{55}{51} &-\frac{75}{34} &-\frac{625}{51}\\ 1&0&0&0&0\\0&1&0&0&0 \end{bmatrix},
 Q_3 = \begin{bmatrix}0&0&1&0&0\\
                      0&0&0&1&0\\
                      0&0&0&0&1\\\end{bmatrix}$$

A set of $100$ points are generated on the cubic scroll $S,$ each of which is obtained as the third intersection of $S$ with a $2$-plane spanned by two points chosen on the centers of projection of the $P_i$ and a third point chosen randomly in $\Pin{4}.$
The result of the experimental process described above is presented  in Figure \ref{scrollcubico}, where the frequency with which the reconstructed solution is close
or far from the true solution $T_P,$ against the values of
$\sigma$ utilized, is plotted. In this case $m = 0.014,$ and $\delta = 0.03.$

\begin{center}
\begin{figure}
\includegraphics[width =0.8\linewidth]{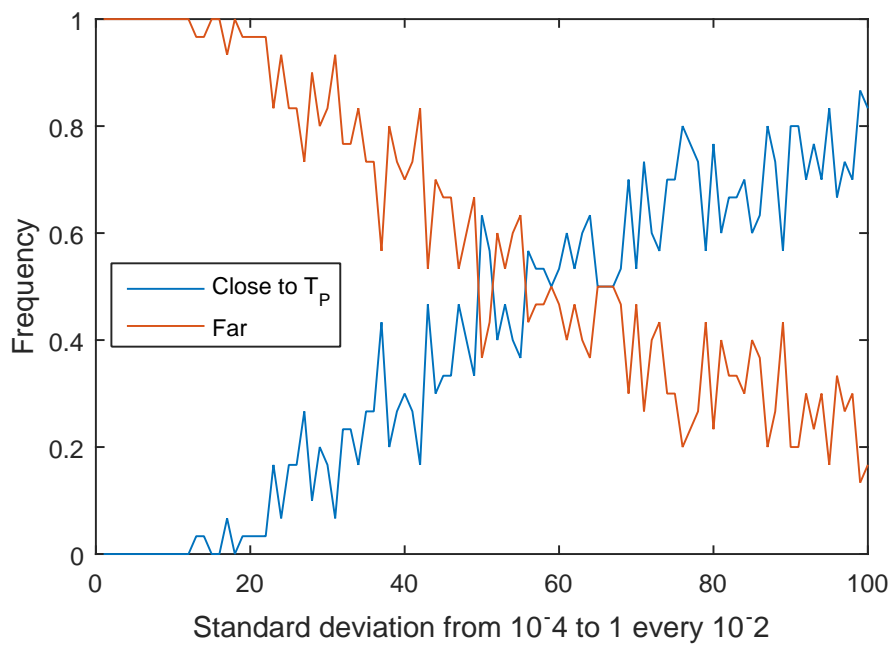}
\label{scrollcubico}
\caption{Instability of reconstruction of a trifocal tensor as described in \ref{inst_scroll} near $S$}
\end{figure}
\end{center}
The set of parameters utilized in this case
is as follows:
\begin{itemize}
\item $X_i \in S$, $\delta = .03,$ $\sigma \in (10^{-4},1),$ every
    $10^{-2}.$
\end{itemize}
The experiment shows that reconstruction near the cubic scroll is quite unstable. As the standard deviation $\sigma$ of the perturbation in $\Pin{4}$ approaches 1 the stability of the reconstruction increases as expected, but  still not fully stable when $\sigma = 1.$

\subsection{Instability results in case $iv$}
\label{inst_cono_quadrico}
In this case, let us choose projection matrices as follows:
$$\scriptstyle
P_1 =  \begin{bmatrix} 0&1 &-6 &-2 &4\\ -1& 0& -15& -5 &10\\0&0&-3&-1&2\end{bmatrix},\ \
 P_2 = \begin{bmatrix}-35&-42 &-47 &13 &-32\\ -8&-12&-11&4&-8\\-2&-3&-2&1&-2 \end{bmatrix},
 P_3= \begin{bmatrix}-2&3 &-12 &6 &3\\ 4&6&0&-8&6\\2&3&0&-2&3 \end{bmatrix}$$
$$\scriptstyle
 Q_1 = \begin{bmatrix} 2&1 & -4&1 &-2\\ 5& -2& 8& 0 &0\\1&0&0&0&0\end{bmatrix},\ \
 Q_2 = \begin{bmatrix}0&5 &-4&3&-6 \\ 0&1&0&0&0\\0&0&1&0&0 \end{bmatrix},
 Q_3 = \begin{bmatrix}0&1&-4&-2&5\\
                      0&0&0&1&0\\
                      0&0&0&0&1\\\end{bmatrix}.$$
Recalling \brref{NfromM}, \brref{emme}, and \brref{A_and_C_case_iv}, one can check that this case is of type $iv$ with $$S_1 = \begin{bmatrix}
0   & -x_3 & x_2& 0 & 0 &x_4\\
x_3 & 0 & -x_1& 0 & x_4 & 0\\
-x_2 & x_1 & 0 & x_4& 0 & 0\\
0 & 0 & 0& -x_3 & -x_2& -x_1
\end{bmatrix},
X_1 = \begin{bmatrix}
0   & 6 &  0 \\
0   & -3 &  0 \\
1   & 0 &  0 \\
0   & -3 &  12 \\
0   & 0 &  0 \\
0   & 0 &  4
 \end{bmatrix}.$$ The equation of the critical quadric cone $Q$ is $x_1^2-2x_1x_2+3x_3x_1+x_1x_4-6x_3x_2.$
A set of $100$ points are generated on the quadric cone, leveraging the fact that the equation of the quadric can be easily parametrized.
The result of the experimental process described above
is presented in Figure \ref{quadric_cone}, where the frequency with which the
reconstructed solution is close or far from the true solution
$T_P,$ against the values of $\sigma$ utilized, is plotted. In
this case $m = 0.0012,$ and $\delta = 0.015.$
\begin{center}
\begin{figure}
\includegraphics[width =0.8\linewidth]{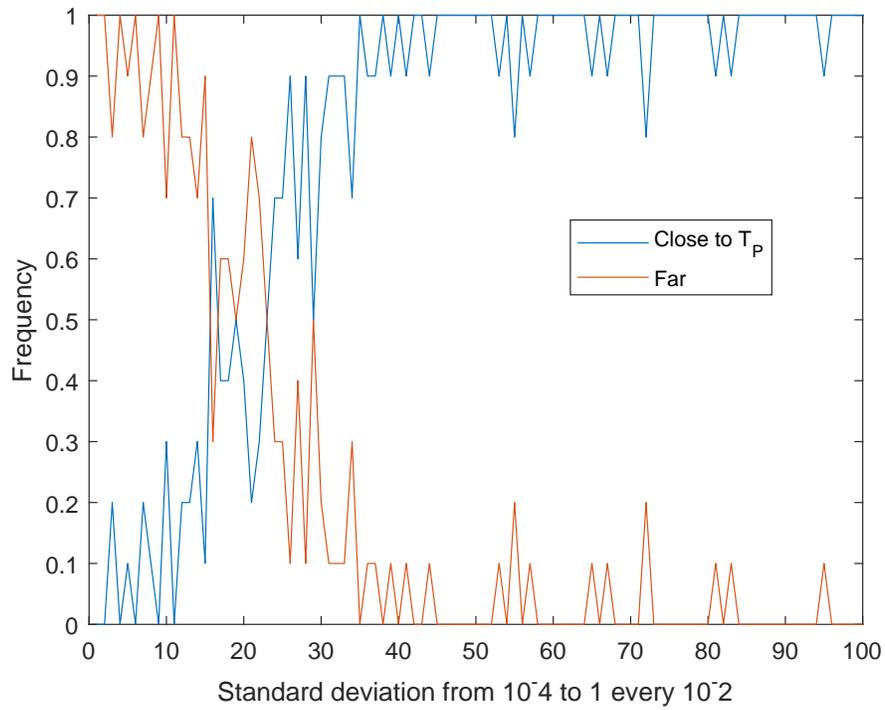}
\label{quadric_cone}
\caption{Instability of reconstruction of a
trifocal tensor as described in \ref{inst_cono_quadrico} near $Q.$}
\end{figure}
\end{center}

\subsection{Instability results in case $v$}
\label{inst_quadric}
In this case, suitable projection matrices are chosen as follows:
$$P_1 = \begin{bmatrix} 1&0 &0 &0 &0\\ 0& 1& 0& 0 &0\\0&0&1&0&0\end{bmatrix},
 P_2 = \begin{bmatrix}0&0 &0 &1 &0\\ 0&0&0&0&1\\1&1&1&0&0 \end{bmatrix},
 P_3 = \begin{bmatrix}0&1 &0 &0 &0\\ 0&0&1&0&0\\1&1&1&1&0 \end{bmatrix}$$
$$Q_1 = \begin{bmatrix} -1&0 & 0&0 &0\\ 0& 0& -1& 0 &0\\0&0&0&-1&0\end{bmatrix},
 Q_2 = \begin{bmatrix}0&0 &0&0&-1 \\ 1&0&0&0&0\\0&1&0&0&0 \end{bmatrix},
 Q_3 = \begin{bmatrix}0&0&1&0&0\\
                      0&0&0&1&0\\
                      0&0&0&0&1\\\end{bmatrix}$$
Various experiments on random choices give the results in the
Figure \ref{quadric} or similar to it.
\begin{center}
\begin{figure}
\includegraphics[width =0.8\linewidth]{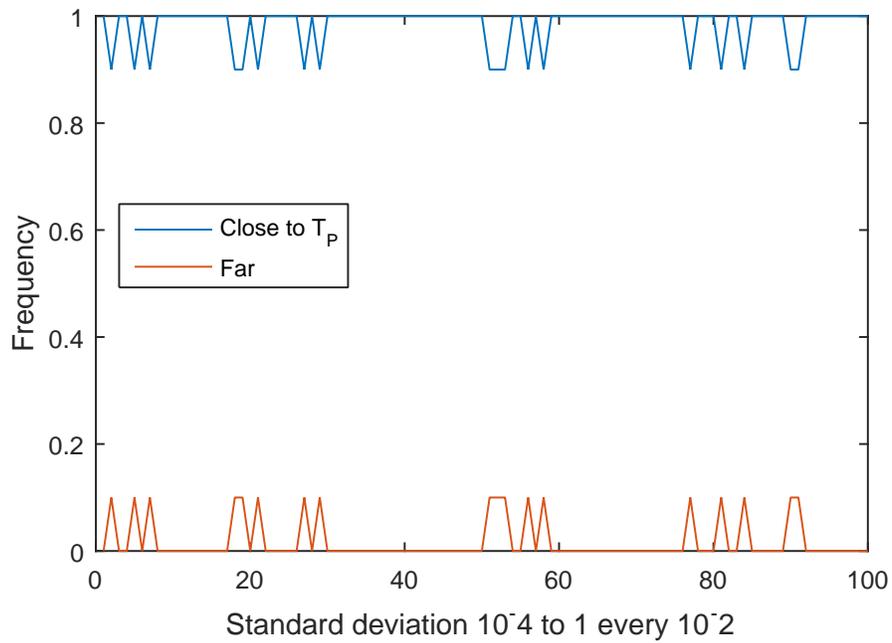}
\label{quadric}
\caption{Instability of reconstruction of a
trifocal tensor as described in \ref{inst_quadric} near $Q.$}
\end{figure}
\end{center}

In all cases, the experiments were performed as follows.
A set of $99$ points are generated on the quadric hypersurface
$Q,$ each of which is obtained as the second intersection of $Q$
with a line through a point chosen on one of the centers of
projection of the $P_i$ and a second point chosen randomly in
$\Pin{4}.$ The result of the experimental process described above
is presented below in Figure 2, where the frequency with which the
reconstructed solution is close or far from the true solution
$T_P,$ against the values of $\sigma$ utilized, is plotted. In
this case $m = 0.015,$ and $\delta = 0.03.$

The set of parameters utilized in this case
is as follows:
\begin{itemize}
\item $X_i \in Q$, $\delta = .03,$ $\sigma \in (10^{-4},1),$ every
    $10^{-2}.$
\end{itemize}
The experiment shows that reconstruction near the quadric is surprisingly very stable.

\newpage

\end{document}